\documentclass[12pt, reqno]{amsart}
\usepackage{lmodern}

\usepackage{amsmath,amsfonts,latexsym,amssymb,enumitem}
\usepackage{comment,url,soul,microtype,color}
\usepackage[backref=page, bookmarks=false, 
colorlinks=true,
linkcolor=blue,
citecolor=blue,
filecolor=blue,
menucolor=blue,
urlcolor=blue]{hyperref}

\calclayout

\let\oldtocsection=\tocsection
\let\oldtocsubsection=\tocsubsection
\let\oldtocsubsubsection=\tocsubsubsection

\renewcommand{\tocsection}[2]{\hspace{0em}\oldtocsection{#1}{#2}}
\renewcommand{\tocsubsection}[2]{\hspace{3em}\oldtocsubsection{#1}{#2}}
\renewcommand{\tocsubsubsection}[2]{\hspace{6em}\oldtocsubsubsection{#1}{#2}}

\newtheorem{lemma}{Lemma}[section]
\newtheorem{proposition}[lemma]{Proposition}
\newtheorem{definition/proposition}[lemma]{Definition/Proposition}
\newtheorem{corollary}[lemma]{Corollary}
\newtheorem{theorem}[lemma]{Theorem}

\theoremstyle{definition}
\newtheorem{definition}[lemma]{Definition}
\newtheorem{remark}[lemma]{Remark}

\newtheorem{question}[lemma]{Question}

\newtheorem{theorema}{Theorem}

\newtheorem{corollarya}[theorema]{Corollary}
\newtheorem{observationa}[theorema]{Observation}
\newtheorem{propositiona}[theorema]{Proposition}

\begin{document}

\title{Contact domination}

\author{Sekh Kiran Ajij}
\address{School of Mathematics \\ Tata Institute of Fundamental Research \\ Mumbai 400005 \\ India}
\email{sekh@math.tifr.res.in}

\author{Ritwik Chakraborty}
\address{School of Mathematics \\ Tata Institute of Fundamental Research \\ Mumbai 400005 \\ India}
\email{ritwik@math.tifr.res.in}

\author{Balarka Sen}
\address{School of Mathematics \\ Tata Institute of Fundamental Research \\ Mumbai 400005 \\ India}
\email{balarka2000@gmail.com, balarka@math.tifr.res.in}

%\subjclass[]{}

%\keywords{}

%\thanks{}

\begin{abstract}
In this note, we prove that every closed connected oriented odd-dimensional manifold admits a map of non-zero degree (i.e., a domination) from a tight contact manifold of the same dimension. This provides an odd-dimensional counterpart of a symplectic domination result due to Joel Fine and Dmitri Panov \cite{FP}. We prove that the dominating contact manifold can be ensured to be Liouville-fillable, but not Weinstein-fillable in general. We discuss an application for contact divisors arising as zero sets of asymptotically contact-holomorphic sections.
\end{abstract}

\maketitle

\tableofcontents

\section{Introduction}

\begin{definition}
Let $M, N$ be a pair of closed connected oriented manifolds  of the same dimension. $M$ \emph{dominates} $N$ if there exists a map $f : M \to N$ of strictly positive degree. 
\end{definition}

A systematic study of the behaviour of various classes of manifolds with respect to this ordering was first suggested by Gromov. We refer the reader to the survey \cite{dlH} for a detailed summary of the basic properties of this ordering, as well as several examples. In this article, we take the point of view that if a certain class $\mathcal{C}$ of manifolds dominates all closed connected oriented manifolds then, in some sense, the class $\mathcal{C}$ is ``large".

Recently, it was shown by Fine and Panov \cite{FP} that every closed connected oriented manifold of even dimension is dominated by a symplectic manifold. This result contrasts with the case of K\"{a}hler manifolds. For instance, it follows from Siu's rigidity theorem \cite{Siu} that compact real hyperbolic manifolds of dimension greater than $2$ are not dominated by K\"{a}hler manifolds. In fact, it was proved by Carlson and Toledo \cite{CaTo} that if $M$ is a K\"{a}hler manifold and $N$ is a locally symmetric space of noncompact type, then $M$ dominates $N$ only if $N$ is locally Hermitian symmetric. These suggest that the class of symplectic manifolds is ``larger" compared to the class of K\"ahler manifolds.

The main aim of this short note is to prove an odd-dimensional counterpart of the main theorem of \cite{FP}:

\begin{theorema}\label{thm-main}
Let $M^{2n+1}$ be a closed connected oriented odd-dimensional manifold. There exists a closed connected oriented manifold $Y^{2n+1}$ supporting a Liouville-fillable (in particular, tight) contact structure and a map $f : Y \to M$ of strictly positive degree.
\end{theorema}

We remark that in dimension $3$, domination by tight contact manifolds follows from known results. Indeed, Sakuma \cite{Sakuma} showed every closed oriented $3$-manifold admits a branched cover which fibers over the circle (see also, \cite{Montesinos}). Using the work of Eliashberg and Thurston \cite{ET}, it follows that fibered $3$-manifolds admit tight contact structures. 

It is not difficult to show that any odd-dimensional closed connected oriented manifold is dominated by \emph{some} contact manifold, given the existence $h$-principle for contact structures in all dimensions proved by Borman, Eliashberg and Murphy \cite{BEM}. However, the contact structures constructed by \cite{BEM} are apriori overtwisted. Thus, the main difficulty is ensuring that the dominating contact manifold is tight. In fact, we show that we can arrange it to be Liouville-fillable. This shows Liouville-fillable contact manifolds abound. We observe that the same is not true for the class of Weinstein-fillable contact manifolds:

\begin{observationa}[Corollary \ref{cor-steinnotdom}]\label{obs-b}
    Let $M^{2n+1}$ be a rationally essential (for instance, negatively curved) manifold. Then $M$ is not dominated by any Weinstein-fillable contact manifold.
\end{observationa}

We call a manifold $M$ rationally $k$-connected if $\pi_1(X)$ is finite, and $\pi_i(X) \otimes \mathbf{Q} = 0$ for all $2 \leq i \leq k$. As a positive result regarding Weinstein-fillable domination, we show: 

\begin{propositiona}[Proposition \ref{prop-ratconnweindom}]\label{prop-c}
    A closed connected oriented rationally $(n-1)$-connected manifold $M^{2n+1}$ is dominated by a Weinstein-fillable contact manifold.
\end{propositiona}

Additionally, as an application of Observation \ref{obs-b} we prove that the asymptotically contact-holomorphic divisors constructed in the work of Ibort, Mart\'{i}nez-Torres and Presas \cite{IMP} are not necessarily Weinstein-fillable:

\begin{corollarya}[Proposition \ref{prop-dondivnotfill}] \label{cor-d} There exists contact manifolds $(Y^{2n+1}, \xi)$ and non-trivial complex line bundles $L$ over $Y$ such that the asymptotically contact-holomorphic divisors of $Y$ corresponding to $L$ are not Weinstein-fillable. \end{corollarya}

This contrasts with the case when $L \cong \underline{\mathbf{C}}$ is the trivial line bundle, i.e.~ when the divisor is homologically trivial. In this case, an observation due to Giroux and Mohsen \cite{Giroux} (see Proposition \ref{prop-girouxobd}) demonstrates that such a divisor is always Weinstein-fillable.

\subsection*{Acknowledgements} The authors thank Mahan Mj and Dishant Pancholi for their interest and encouragement. The third author thanks Aditya Kumar and Mike Miller Eismeier for various discussions and comments during the preparation of this article. This work is supported by the Department of Atomic Energy, Government of India, under project no.12-R\&D-TFR-5.01-0500.

\section{Preliminary definitions and properties}\label{sec-prel}

\subsection{Liouville and Weinstein domains} 

\begin{definition}\cite{Yasha}
A \emph{Liouville domain} is a tuple $(W, \omega, v)$ where 
$W$ is a compact manifold with boundary, $\omega$ is an {exact} symplectic form on $W$, and $v$ is a \emph{Liouville vector field} for $(W, \omega)$, i.e., $v$ is a vector field on $W$ pointing transversely out of the boundary $\partial W$, such that $\mathcal{L}_v \omega = \omega$. 

A \emph{Weinstein domain} is a tuple $(W, \omega, v, \phi)$ such that $(W, \omega, v)$ is a Liouville domain, and $\phi : W \to \mathbf{R}$ is a Morse function which is \emph{gradient-like} with respect to $v$, i.e., for some some $\delta > 0$ and some choice of an ambient metric on $W$, one has $df(v) \geq \delta (\|v\|^2 + \|df\|^2)$.
\end{definition}

We record the following simple lemma for future use:

\begin{lemma}\label{lem-prod}
Let $(W_1, \omega_1, v_1)$ and $(W_2, \omega_2, v_2)$ be a pair of Liouville domains. The product manifold $W_1 \times W_2$ admits the structure of a Liouville domain.
\end{lemma}

\begin{proof}
Note $W_1 \times W_2$ is a manifold with boundaries and corners. Nevertheless, $\omega := \omega_1 \oplus \omega_2$ defined an exact symplectic form in the interior of $W_1 \times W_2$. Consider the vector field $v := v_1 \oplus v_2$ defined in the interior of $W_1 \times W_2$. Then,
$$\mathcal{L}_{v} \omega = \mathcal{L}_{v_1} \omega_1 \oplus \mathcal{L}_{v_2} \omega_2 = \omega_1 \oplus \omega_2 = \omega.$$
We smooth out the corners of $W_1 \times W_2$ by deleting an open $\varepsilon$-neighborhood of the union of boundaries and corners of $W_1 \times W_2$, for some small $\varepsilon > 0$. The resulting manifold with boundary is homeomorphic to $W_1 \times W_2$. Moreover, the vector field $v = v_1 \oplus v_2$ points transversely out of the resulting boundary. This proves the claim.
\end{proof}

\begin{definition}
    A contact manifold $(Y, \xi)$ is \emph{Liouville-fillable} (resp. \emph{Weinstein-fillable}) if there exists a Liouville (resp. Weinstein) domain $(W, \omega, v)$ (resp. $(W, \omega, v, \phi)$) such that $Y = \partial W$ and $\xi = \ker(i_v \omega|_{\partial W})$.
\end{definition}

The following theorem is a consequence of \cite{Niederkruger} and the discussion in \cite[pg.~4]{BEM}:

\begin{theorem}\label{thm-filltight} Liouville-fillable contact manifolds are tight.
\end{theorem}

\subsection{Symplectic and contact divisors} In this section, we shall summarize some foundational results pertaining to the existence of symplectic and contact divisors that we shall be using in the rest of the article. We begin with the following definition: 

\begin{definition}
    A symplectic manifold $(X, \omega)$ is said to be \emph{integral} if $[\omega]$ is contained in the image of the change of coefficients homomorphism $H^2(X; \mathbf{Z}) \to H^2(X; \mathbf{R})$.
\end{definition}

\begin{remark} \label{rem-integral}
    Given a symplectic manifold $(X, \omega)$, one can always find a symplectic form $\omega_0$ on $X$ such that $(X, \omega_0)$ is integral. To see this, let us choose a basis of $H^2(X; \mathbf{R})$, and write $\omega$ as a $\mathbf{R}$-linear combination of closed $2$-forms representing the basis elements. Note that $H^2(X; \mathbf{Q}) \subset H^2(X; \mathbf{R})$ is dense. Therefore, by a small perturbation of the coefficients, we may find a new form $\omega'$ such that $[\omega'] \in H^2(X; \mathbf{Q})$. Since $\omega'$ is a $\mathbf{R}$-linear combination of closed $2$-forms, it is closed. Moreover, since $\omega'$ is $C^\infty$-close to $\omega$, it is non-degenerate. By multiplying $\omega'$ by sufficiently large integer, we can then find a $2$-form $\omega_0$ such that $(X, \omega_0)$ is an integral symplectic manifold.
\end{remark}

\begin{definition}Given an integral symplectic manifold $(X, \omega)$, the \emph{pre-quantum bundle} $L$ over $X$ is a complex line bundle with Chern class $c_1(L) = [\omega]$.
\end{definition}

The following theorem due to Donaldson is a foundational result in symplectic geometry:

\begin{theorem}\cite{Don} \label{thm-don}
$(X^{2n}, \omega)$ be an integral symplectic manifold. For all sufficiently large $k \gg 1$, there exists a codimension $2$ symplectic submanifold $Z \subset X$ such that $[Z] = \mathrm{PD}(k[\omega])$.
\end{theorem}

We shall call codimension $2$ symplectic submanifolds $Z \subset (X, \omega)$ appearing in the statement of Theorem \ref{thm-don} as \emph{Donaldson divisors}. The contact analogue of Theorem \ref{thm-don} is due to Ibort, Mart\'{i}nez-Torres and Presas:

\begin{theorem}\cite[Theorem 1]{IMP} \label{thm-imp} $(Y^{2n+1}, \xi)$ be a contact manifold and $E \to Y$ be a complex vector bundle of rank $r \leq n$. Then there exists a contact submanifold $Z \subset Y$ with $[Z] = \mathrm{PD}(c_r(E)) \in H_{2n+1-2r}(Y)$. Moreover, the inclusion $Z \to Y$ induces an isomorphism on the homotopy groups $\pi_i$ for all $1 \leq i \leq n-r-1$ and a surjection on $i = n-r$. \end{theorem}

For the sake of completeness and context for some discussions in Section \ref{sec-app}, we summarize here the proof of Theorem \ref{thm-imp} following \cite{IMP}:

\begin{proof}[Summary of Proof] Let $\alpha$ be a contact $1$-form such that $\xi = \ker \alpha$. Let $J$ be a fiberwise almost complex structure on $\xi$ compatible with the fiberwise symplectic form $d\alpha$. Let $R_\alpha$ denote the Reeb vector field associated to $\alpha$, and $g$ be the metric on $M$ such that $g(v, w) = d\alpha(v, Jw)$ for all $v, w \in \xi$ and $R_\alpha$ is the unit oriented orthonormal vector to $\xi$ with respect to $g$.

Let $L_0 := M \times \mathbf C$ denote the trivial line bundle on $M$, equipped with the connection $1$-form $-i\alpha$. Therefore, the tensor power $L_0^{\otimes k}$ is the trivial bundle naturally equipped with the connection $1$-form $-ik \alpha$. We pick an arbitrary connection operator $\nabla^E$ on $E$. We denote $E_k := L \otimes L_0^{\otimes k}$ and let $\nabla_k$ be the connection operator on $E_k$ defined by $\nabla_k := \nabla^E - ik\alpha$. Let $\partial_k$ (resp. $\overline{\partial}_k$) denote the $J$-linear (resp. $J$-antilinear) parts of $(\nabla_k s_k)|_{\xi}$.

The construction of \cite{IMP} produces, for any given constants $C, \eta > 0$, a family of sections $s_k : M \to E_k$ ($k \geq 1$) such that:

\begin{enumerate}
\item $\{s_k : k \geq 1\}$ is \emph{asymptotically contact-holomorphic}, i.e., for all $p \in M$, 
$$|s_k(p)| \leq C, \, |\nabla_k s_k(p)| \leq C \sqrt{k} \, \text{ and } \, |\overline{\partial}_k s_k(p)| \leq C.$$
\item $\{s_k : k \geq 1\}$ is \emph{equitransverse}, i.e., for all $p \in M$ such that $|s_k(p)| \leq \eta$,
$$|\partial_k s_k(p)| \geq \eta \sqrt{k}.$$
\end{enumerate}
Condition $(2)$ ensures that $s_k$ is transverse to the zero section $\mathbf{0} \subset E_k$. Thus, $Z_k := s_k^{-1}(\mathbf{0}) \subset M$ defines a codimension-$2$ submanifold. By construction $c_1(E) = c_1(E_k)$ is Poincar\'{e} dual to $[Z_k]$. Furthermore, $(Z_k, \xi|_{Z_k})$ is a contact manifold for $k \gg 1$. Indeed, 
$$T_k Z_p \cap \xi_k = \ker (\nabla_k s_k)|_{\xi_k} = \ker (\overline{\partial}_k s_k + \partial_k s_k).$$
By Condition $(1)$ and $(2)$, $|\overline{\partial}_k s_k| \leq C\eta^{-1} |\partial_k s_k|/\sqrt{k}$. This ensures for sufficiently large $k \gg 1$, $T_k Z_p \cap \xi_k$ is close (in the Grassmannian) to a $J$-complex subspace of $(\xi_k, d\alpha_k, J)$, hence it is symplectic. See \cite[Lemma 4]{IMP} for more details. The statement regarding homotopy groups of $Z_k$ follows from \cite[Section 5.1]{IMP} which shows $-\log |s_k|^2 : M \setminus Z_k \to \mathbf R$ defines a proper Morse function on $M \setminus Z_k$ with all critical points having at most $n - r$.
\end{proof}

The following result pertaining to the complement of Donaldson divisors is due to Giroux \cite{Giroux_don}:

\begin{theorem}\cite[Theorem 2]{Giroux_don} \label{thm-girdon}
    Let $(X^{2n}, \omega)$ be an integral symplectic manifold. Let $Z \subset (X, \omega)$ be a Donaldson divisor (see, Theorem \ref{thm-don}). Let $\nu(Z)$ be a tubular neighborhood of $Z$ in $X$. Then, $(X \setminus \nu(Z), \omega)$ admits the structure of a Weinstein domain.
\end{theorem}

\begin{definition}
A symplectic manifold $(M, \omega)$ with non-empty boundary $\partial M \neq \emptyset$ is \emph{convex} if there exists a vector field $v$ defined in a collar neighborhood of $\partial M \subset M$ such that $v$ points transversely outwards of $\partial M$, and $\mathcal{L}_v \omega = \omega$. In this case, the distribution $\xi := \ker(i_v \omega|_{\partial M})$ defines a contact structure on $\partial M$. 
\end{definition}

\begin{remark}
    A Liouville manifold $(W, \omega, v)$ is an example of a convex symplectic manifold with contact boundary $(\partial W, \xi)$, where $\xi = \ker(i_v \omega)$.
\end{remark}

A key technical ingredient in our article will be the following relative version of Theorem \ref{thm-don} for symplectic manifolds with boundary, due to Presas \cite{Presas}:

\begin{theorem}\cite[Theorem 1.1]{Presas}\label{thm-presas}
Let $(M, \omega)$ be a convex symplectic manifold which is integral, with prequantizable bundle $L$ and with contact boundary $(C, \xi)$. Fix a rank $r$ complex vector bundle $E$ over $M$. For all sufficiently large $k \gg 1$, there exists a symplectic submanifold $W$ of $M$ {transverse to C}, which is Poincar\'{e} dual to $c_r(L^{\otimes k} \otimes E)$, satisfying that $W \cap C$ is a contact submanifold of $(C, \xi)$.
\end{theorem}

\section{Proof of Theorem \ref{thm-main}}\label{sec-mainthm}

Given an oriented manifold with boundary $(W, \partial W)$, we define the \emph{double} of $W$ by gluing $W$ with its orientation reversed copy $\overline{W}$ along the boundary by identity. The resulting manifold will be denoted as $D(W) := W \cup_{\partial} \overline{W}$. We begin with the following observation.

\begin{proposition}\label{prop-doublecircle}
Let $(W, \omega, v)$ be a Liouville domain. Then the manifold $D(W) \times S^1$ admits a Liouville-fillable (hence, tight) contact structure.
\end{proposition}

\begin{proof}
Consider the annulus $S^1 \times [-1, 1] \cong D^2_{2}(0) \setminus D^2_{0.5}(0) \subset \mathbf{R}^2$ with the exact symplectic form $\omega = r dr \wedge d\theta$. Let us define a vector field 
$$v_0 := \frac{1}{2} \left ( r - \frac{1}{r} \right ) \partial_r$$
Then, we calculate: 
$$i_{v_0} \omega = \frac{r^2}{2} d\theta - \frac{1}{2} d\theta$$
Hence, $\mathcal{L}_{v_0} \omega = d i_{v_0} \omega = \omega$, since $d\theta$ is closed. Moreover, $v_0$ points transversely out of both the inner as well as the outer boundary of the annulus $D^2_{2}(0) \setminus D^2_{0.5}(0)$. Thus, $(S^1 \times [-1, 1], \omega, v_0)$ is a Liouville domain. The product $W \times (S^1 \times [-1, 1])$ also admits the structure of a Liouville domain, by Lemma \ref{lem-prod}. Observe, 
$$\partial (W \times (S^1 \times [-1, 1])) = D(W) \times S^1.$$
Therefore, the manifold $D(W) \times S^1$ admits a Liouville-fillable contact structure. By Theorem \ref{thm-filltight}, $D(W) \times S^1$ admits a tight contact structure.
\end{proof}

\begin{proposition}[Contact ascent]\label{prop-contasc}
Let $M^{2n-1}$ be an odd-dimensional closed connected oriented manifold.  There exists a Liouville-fillable contact manifold $Y^{2n+1} = \partial P$ with Liouville filling $P$ and a map $f : Y \to M$ such that for any regular fiber $F := f^{-1}(p)$, $[F] \in H_2(P; \mathbf{Z})$ is non-torsion.
\end{proposition}

\begin{proof}

Note that $M \times S^1$ is a closed connected oriented manifold of even dimension $2n$. By \cite[Theorem 1]{FP}, there exists a symplectic manifold $X^{2n}$ and a positive degree map $g : X \to M \times S^1$. By Remark \ref{rem-integral}, we may choose an integral symplectic form $\omega$ on $X$. Let $\Sigma^{2n-2} \subset X$ be a Donaldson divisor of $(X, \omega)$. Let $\nu(\Sigma)$ denote a $\varepsilon$-neighborhood of $\Sigma \subset X$. By a slight abuse of notation, we shall identify $\nu(\Sigma)$ with the unit disk bundle of the normal bundle of $\Sigma \subset X$. Let $\pi : \nu(\Sigma) \to \Sigma$ denote the normal bundle projection.

Let $W := X \setminus \nu(\Sigma)$. Then $W$ is a Liouville domain (in fact, it is a Weinstein domain, by Theorem \ref{thm-girdon}). By Lemma \ref{lem-prod}, $W \times D^2$ is also a Liouville domain. Let $Z = \partial (W \times D^2)$. Observe,
\begin{equation}\label{eq-decZ}Z = W \times S^1 \cup \partial W \times D^2\end{equation}
Let $Y = D(W) \times S^1$, and $P := W \times (S^1 \times [-1, 1])$. By Proposition \ref{prop-doublecircle}, $P$ is a Liouville filling of $Y$. We shall define the map $f : Y \to M$ as a composition of several maps:
$$f : Y \stackrel{k}{\to} Z \stackrel{h}{\to} X \stackrel{g}{\to} M \times S^1 \stackrel{\pi_M}{\to} M$$
The map $\pi_M : M \times S^1 \to M$ is projection to the first factor, and $g : X \to M \times S^1$ is the symplectic domination defined earlier. It remains to define the maps $k$ and $h$. 

\begin{enumerate}[topsep=1ex,itemsep=1ex,partopsep=1ex,parsep=1ex, wide=0pt, font=\itshape]
\item[Construction of $k : Y \to Z$:] Let $\tau_0$ be the orientation-reversing involution on $D(W) = W \cup_\partial \overline{W}$ defined by reflecting along the common boundary $\partial W = \partial \overline{W}$. We define,
\begin{gather*}
    \tau : Y \to Y, \\ \tau(w, z) = (\tau_0(w), \overline{z})
\end{gather*}
Note that $\tau$ is an orientation-preserving involution of $Y = D(W) \times S^1$, as it reverses the orientation of both factors $D(W)$ and $S^1$ individually. Writing $D(W) = W \cup \partial W \times [-1, 1] \cup \overline{W}$, we see $\tau_0$ acts on $D(W)$ by exhanging the first and third components and on the bi-collar $W \times [-1, 1]$ by $\tau_0(w, t) = (w, -t)$. Therefore, $\tau$ acts on 
$$Y = D(W) \times S^1 = (W \times S^1) \cup \partial W \times (S^1 \times [-1, 1]) \cup (\overline{W} \times S^1),$$
by exchanging the first and third components and on $W \times (S^1 \times [-1, 1])$ by $\tau(w, (z, t)) = (w, \overline{z}, -t)$. The quotient of $S^1 \times [-1, 1]$ by the involution $(z, t) \mapsto (\overline{z}, -t)$ is $D^2$. Therefore, 
$$Y/\tau \cong W \times S^1 \cup \partial W \times D^2 \cong Z$$
We define $k : Y \to Y/\tau \cong Z$ as the quotient map.

\item[Construction of $h : Z \to X$:] Recall $\nu(\Sigma)$ is the unit disk bundle of the normal bundle of $\Sigma \subset X$. The normal bundle possesses a fiberwise $\mathbf{R}^\times$-action by scaling. We begin by defining the following map:
\begin{gather*}\Phi : \partial \nu(\Sigma) \times [0, 1] \to \nu(\Sigma),\\
\Phi(x, r) = 
\begin{cases} 
      \pi(x), & \text{if}\;\; r = 0 \\
      r \cdot x, & \text{if} \;\; r \in (0, 1]\\
   \end{cases}
\end{gather*}
Note that $\Phi(x, r)$ belongs to the circle bundle of radius $r$ in $\nu(\Sigma)$ for $r \neq 0$. Next, we define $h : Z \to X$ piecewise with respect to the decomposition (\ref{eq-decZ}):
\begin{gather*}h : Z = W \times S^1 \cup \partial W \times D^2 \to W \cup \nu(\Sigma) = X,\\
h(z) = \begin{cases} 
      w, & \text{if}\;\; z = (w, \theta)\in W \times S^1 \\
      \Phi(w, r), & \text{if} \;\; z = (w, (r, \theta)) \in \partial W \times D^2\\
   \end{cases}
\end{gather*}
Here, we use $\partial W = \partial \nu(\Sigma)$. Note that for any $z = (w, 1) \in \partial W = \partial \nu(\Sigma),$ $\Phi(w, 1) = 1 \cdot w = w$. Therefore, the map $h$ is continuous by the pasting lemma.
\end{enumerate}

We have defined the map $f : Y \to M$. We now proceed to show that a regular fiber of $f$ is a (possibly disconnected) sub-surface in $P$ representing a non-zero rational homology class. Since $\dim \Sigma < \dim M$ and $\pi_M \circ g : X \to M$ is a smooth map, the image $(\pi_M \circ g)(\Sigma) \subset M$ is a proper subset. Choose a point $p \in M \setminus (\pi_M \circ g)(\Sigma)$. Then, $\pi_M^{-1}(p) = \{p\} \times S^1 \subset M \times S^1$. Let us homotope $g$ slightly to make it transverse to $\{p\} \times S^1$. Thus, 
$$g^{-1}(\{p\} \times S^1) = \gamma_1 \sqcup \gamma_2 \sqcup \cdots \sqcup \gamma_\ell,$$
where $\{\gamma_i\} \subset X$ are a disjoint collection of simple closed curves. By our choice of $p \in M$, $\{p\} \times S^1$ intersects $g(\Sigma) \subset M \times S^1$ trivially. Therefore, each $\gamma_i$ is disjoint from the Donaldson divisor $\Sigma \subset X$. Thus, $\gamma_i \subset X \setminus \nu(\Sigma) = W$. From the construction of $h : Z \to X$, we have
$$h^{-1}(\gamma_i) = \gamma_i \times S^1 \subset W \times S^1$$ 
Further, from the construction of $k$, we have
$$k^{-1}(\gamma_i \times S^1) = \gamma_i \times S^1 \sqcup \tau(\gamma_i \times S^1) \subset Y$$
Therefore,
\begin{align*}
f^{-1}(p) &= \bigsqcup_i (\gamma_i \times S^1 \sqcup \tau(\gamma_i \times S^1)) \subset Y\\
    &= \bigsqcup_i (\gamma_i \times S^1 \times \{-1\}) \sqcup (\gamma_i \times S^1 \times \{1\}) \subset P = W \times S^1 \times [-1, 1]
\end{align*}
Let $\alpha = [g^{-1}(\{p\} \times S^1)] = [\gamma_1] + \cdots + [\gamma_\ell] \in H_1(X; \mathbf{Q})$. Then, $\alpha = g_!([\{p\} \times S^1])$ where $g_! : H_1(M \times S^1; \mathbf{Q}) \to H_1(X; \mathbf{Q})$ is the Umkehr map in homology\footnote{For oriented connected closed manifolds $M, N$ of the same dimension and a map $f : M \to N$, the Umkehr (``wrong way") map in homology is a morphism $f_! : H_k(N) \to H_k(M)$, such that for any $k$-dimensional closed smooth submanifold $S \subset N$ transverse to $f$, one has $f_![S] = [f^{-1}(S)]$. See, \cite[Definition 2.4]{dlH}.}. Note that $g_!$ is Poincar\'{e} dual to $g^* : H^1(M \times S^1; \mathbf{Q}) \to H^1(X; \mathbf{Q})$. Since $\mathrm{deg}(g) \neq 0$, $g^*$ must be injective (see, \cite[Proposition 2.6]{dlH}). Therefore, $\alpha \neq 0$.  Since $\gamma_i \subset W$, $\alpha$ represents a class in $H_1(W; \mathbf{Q})$ as well. This class must be non-zero, as being null-homologous in $W$ would force being null-homologous in $X$. Finally, note that
$$[\gamma_i \times S^1] = [\gamma_i] \otimes [\{p\} \times S^1] \in H_1(W \times S^1)$$
Therefore, 
$[f^{-1}(p)] = 2 \alpha \otimes [\{p\} \times S^1] \in H_1(W \times S^1 \times [-1, 1]; \mathbf{Q}) = H_1(P; \mathbf{Q})$. Since $\alpha \neq 0$, we obtain, $[f^{-1}(p)] \neq 0 \in H_1(P; \mathbf{Q})$. In other words, $[f^{-1}(p)] \in H_1(P; \mathbf{Z})$ is non-torsion. This concludes the proof.
\end{proof}

We are now ready to prove Theorem \ref{thm-main}.

\begin{proof}[Proof of Theorem \ref{thm-main}]
Let $M^{2n-1}$ be an arbitrary odd-dimensional manifold. By Proposition \ref{prop-contasc}, we obtain a Liouville manifold $(P^{2n+2}, \omega)$ with contact boundary $Y^{2n+1}$ and a map $f : Y \to M$ such that for a regular fiber $F = f^{-1}(p)$, the class $[F] \in H_2(P, \mathbf{Z})$ is not torsion. Therefore, the dual $\alpha \in H^2(P, \mathbf{Z})$ of $[F]$ under the universal coefficient theorem is non-zero. Let $E$ be a complex line bundle over $P$ with $c_1(E) = \alpha$. Let $L$ be the pre-quantum bundle over $P$. Since $P$ is Liouville, $c_1(P) = [\omega] = 0$. Therefore, $L$ is a trivial complex line bundle over $P$. Thus, $L^{\otimes k} \otimes E \cong E$. By Theorem \ref{thm-presas}, there exists a symplectic submanifold $W \subset P$ which is Poincar\'{e} dual to $\alpha$ and $W \cap Y \subset Y$ is a contact submanifold. Therefore, we have:
\begin{enumerate}
    \item $W$ is a convex symplectic manifold with boundary $W \cap C$,
    \item $W \cap Y \subset Y$ is Poincar\'{e} dual to $\alpha \in H^2(Y, \mathbf{Z})$.
\end{enumerate}
Consider the restriction $f|_{W \cap Y} : W \cap Y \to M$. By $(1)$, $W \cap Y$ is a Liouville-fillable (hence, tight) contact manifold. By $(2)$, the oriented intersection number of the regular fiber $F$ of $f : Y \to M$ and $W \cap Y$ is positive. Hence, $f|_{W \cap Y} : W \cap Y \to M$ has positive degree.  \end{proof}

\section{Weinstein-fillable domination}\label{sec-weinfill}

We recall the following notion introduced by Gromov in the context of the systolic inequality:

\begin{definition}\cite{Gromov} An oriented closed connected manifold $M^m$ with fundamental group $\pi = \pi_1(M)$ is \emph{essential} if $\phi_*([M]) \neq 0 \in H_n(K(\pi, 1); \mathbf Z)$ where $\phi : M \to K(\pi, 1)$ is the map classifying the universal cover of $M$. If moreover, $\phi_*([M]) \neq 0 \in H_n(K(\pi, 1); \mathbf Q)$, then we say $M$ is \emph{rationally essential}.\end{definition}

In the next proposition, we make the elementary observation that high dimensional Weinstein-fillable contact manifolds are inessential. We refer the reader to \cite{BCS} for more refined (and in fact, a complete) obstruction to Weinstein fillability in high dimensions.

\begin{proposition}\label{prop-steininess}For $n \geq 3$, a Weinstein-fillable contact manifold $(Y^{2n-1}, \xi)$ is not essential.\end{proposition}

\begin{proof}Let $(W^{2n}, \omega, v, \phi)$ be a Weinstein filling for $Y$. Since the Morse function $\phi$ has critical points of index at most $n$ (see, for instance, \cite{Yasha}), we may use it to construct a handlebody decomposition of $W$ with handles of index at most $n$. Turning the handlebody upside-down, we deduce $W$ is obtained from $Y$ by attaching handles of index at least $n$. If $n \geq 3$, then this implies $Y \hookrightarrow W$ induces an isomorphism at the level of $\pi_1$. Let $\pi := \pi_1(Y) \cong \pi_1(W)$. Since the classifying map of the universal cover is functorial, we obtain that the map $\phi : Y \to K(\pi, 1)$ extends to a map $\Phi : W \to K(\pi, 1)$. Therefore, $\phi_\#[Y] = \partial \Phi_\#[W]$ is a boundary in the singular chain complex of $K(\pi, 1)$. Thus, $\phi_*[Y] = 0 \in H_n(K(\pi, 1); \mathbf Z)$, as desired.\end{proof}

\begin{corollary}\label{cor-steinnotdom}Let $M^{2n+1}$ be a rationally essential manifold. Then $M$ is not dominated by any Weinstein-fillable contact manifold.\end{corollary}

\begin{proof}It is a straightforward observation that any manifold dominating a rationally essential manifold must in turn be rationally essential (see, for instance, \cite[Example 4.5.(3)]{dlH}). The conclusion is then immediate from Proposition \ref{prop-steininess}.\end{proof}

Next, we give a class of manifolds which admit a domination by Weinstein-fillable contact manifolds. We introduce the following definition:

\begin{definition}A manifold $X$ shall be called \emph{rationally $k$-connected} if $\pi_1(X)$ is finite (not necessarily abelian) and $\pi_i(X) \otimes \mathbf Q = 0$ for all $2 \leq i \leq k$.\end{definition} 

In \cite{Geiges}, Geiges showed that any $(n-1)$-connected $(2n+1)$-manifold is diffeomorphic, upto connect sum with exotic spheres, to a manifold that admits a Weinstein-fillable contact structure in every homotopy class of almost contact structures. We prove the following result with the weakened hypothesis of rational $(n-1)$-connectedness:

\begin{proposition}\label{prop-ratconnweindom} A closed connected oriented rationally $(n-1)$-connected manifold $X^{2n+1}$ is dominated by a Weinstein-fillable contact manifold.\end{proposition}

\begin{proof}By passing to the universal cover, we may assume without loss of generality that $X$ is simply connected. Let 
$$h_i : \pi_i(X) \otimes \mathbf Q \to H_i(X; \mathbf Q)$$
denote the rational Hurewicz homomorphism. Since $\pi_i(X) \otimes \mathbf Q = 0$ for $1 < i < n$, $h_i$ is an isomorphism for $1 \leq i < 2n-1$ by the rational Hurewicz theorem \cite{KK}. In particular, $h_n$ and $h_{n+1}$ are isomorphisms. 

By Poincar\'{e} duality, $H_n(X; \mathbf Q) \cong H_{n+1}(X; \mathbf Q)$. Therefore, we may choose bases $\{\alpha_i : 1 \leq i \leq k\}$ (resp. $\{\beta_i : 1 \leq i \leq k\}$) for $H_n(X; \mathbf Q)$ (resp. $H_{n+1}(X; \mathbf Q)$) such that $\langle \mathrm{PD}(\alpha_i), \beta_j \rangle = \delta_{ij}$. Since $h_n$ and $h_{n+1}$ are isomorphisms, we may choose spherical representatives $\varphi_i : S^n \to X$ (resp. $\psi_i : S^{n+1} \to X$) for $\alpha_i$ (resp. $\beta_i$). By taking the wedge sum of these, we define a map
$$f : \vee^k (S^n \vee S^{n+1}) \to X$$
We would like to extend $f$ to a map $F : \#^k (S^n \times S^{n+1}) \to X$. To this end, observe that $\#^k(S^n \times S^{n+1})$ is obtained from attaching a $(2n+1)$-cell to $\vee^k (S^n \vee S^{n+1})$ by an element of $\pi_{2n+1}(\#^k(S^n \times S^{n+1}))$ given by product of $k$ many Whitehead brackets between the generators of $\pi_n(S^n)$ and $\pi_n(S^{n+1})$. Therefore, it certainly \emph{suffices} to show that the element
$$[\varphi_1, \psi_1] \cdots [\varphi_n, \psi_n] \in \pi_{2n}(X)$$
is zero. In fact, we shall show the weaker statement that it is torsion. To accomplish this, we begin by observing that 
$$H^*(X; \mathbf{Q}) \cong H^*(\#^k (S^n \times S^{n+1}); \mathbf{Q}),$$
as graded $\mathbf{Q}$-algebras. Since $X$ and $\#^k(S^n \times S^{n+1})$ are both simply connected and rationally $(n-1)$-connected, by \cite{Miller} they are formal in the sense of Sullivan (see, \cite{Sullivan}). Therefore, their Sullivan minimal models must be the same. Since the rational homotopy graded Lie algebra is completely determined by the Sullivan minimal model, we conclude 
$$\pi_*(X) \otimes \mathbf Q \cong \pi_*(\#^k(S^n \times S^{n+1})) \otimes \mathbf Q$$
as graded Lie algebras over $\mathbf{Q}$. Therefore, we conclude
$$[\varphi_1, \psi_1] \cdots [\varphi_n, \psi_n] = 0\in \pi_{2n}(X) \otimes \mathbf{Q}$$
Suppose $[\varphi_1, \psi_1] \cdots [\varphi_n, \psi_n] \in \pi_{2n}(X)$ is annihilated by some $d > 0$. We define $\varphi_i'$ (resp. $\psi_i'$) by pre-composing $\varphi_i$ (resp. $\psi_i$) with the degree $d$ self-map of the sphere. Taking a wedge sum of these leads to a map $f' : \vee^k (S^n \vee S^{n+1}) \to X$ which extends to a map $F : \#^k (S^n \times S^{n+1}) \to X$. By construction, $F_*([S^n]) = d \cdot \alpha_i$ and $F_*([S^{n+1}]) = d \cdot\beta_i$. Therefore, $F_*([S^n])$ and $F_*([S^{n+1}])$ have non-zero intersection pairing in $X$. Therefore, by Poincar\'{e} duality, 
$$F^* : H^{2n+1}(X; \mathbf{Q}) \to H^{2n+1}(\#^k (S^n \times S^{n+1}; \mathbf{Q})$$ 
is non-zero. Thus, $X$ is dominated by $\#^k (S^n \times S^{n+1})$. 

Note that $\#^k (S^n \times S^{n+1})$ is the boundary of $\#_\partial^k (S^n \times D^{n+2})$, where $\#_\partial$ denotes the boundary connect sum operation. Since $\#_\partial^k (S^n \times D^{n+2})$ can be obtained from $D^{2n+2}$ by attaching $k$ subcritical $n$-handles along $k$ isotropic embeddings $S^n \hookrightarrow (S^{2n+1}, \xi_{\mathrm{std}})$ contained in disjoint Darboux charts, it is a (subcritical) Weinstein manifold. This concludes the proof. \end{proof}

\section{Application}\label{sec-app}

Let us recall the contact submanifolds with prescribed homology class produced by Theorem \ref{thm-imp} due to Ibort, Mart\'{i}nez-Torres and Presas \cite{IMP}. We introduce the following terminology: 

\begin{definition}
    Let $(Y^{2n+1}, \xi)$ be a contact manifold and $L \to Y$ be a complex line bundle. We shall call the codimension $2$ contact submanifold $Z \subset Y$ with $[Z] = \mathrm{PD}(c_1(L))$ produced by Theorem \ref{thm-imp} as an \emph{asymptotically contact-holomorphic divisor} corresponding to $L$.
\end{definition}

The following result is an observation which follows from the work of Giroux and Mohsen \cite{Giroux} on the existence of open book decompositions compatible with contact structures in high dimensions.

\begin{proposition}\label{prop-girouxobd}
    Let $(Y^{2n+1}, \xi)$ be a contact manifold and $L \cong Y \times \underline{\mathbf{C}}$ be a trivial complex line bundle. Then an asymptotically contact-holomorphic divisor corresponding to $L$ is Weinstein-fillable.
\end{proposition}

\begin{proof}
    Recall the sections $s_k : M \to \mathbf{C}$ of $L$ satisfying Conditions $(1)$ and $(2)$ in the summary of the proof of Theorem \ref{thm-imp}, which give rise to the required contact divisors $Z_k = s_k^{-1}(0)$ for $k \gg 1$. By \cite[Theorem 10]{Giroux}, $s_k/|s_k| : M \setminus Z_k \to S^1$ provides a Weinstein open-book structure on $M$ with $Z_k$ as the binding. In particular, $Z_k$ admits a Weinstein-filling by the pages of this open book decomposition of $M$.
\end{proof}

The main result of this section is the following:

\begin{proposition}\label{prop-dondivnotfill}There exists contact manifolds $(Y^{2n+1}, \xi)$ and non-trivial complex line bundles $L$ over $Y$ such that the asymptoticaly contact-holomorphic divisors of $Y$ corresponding to $L$ are not Weinstein-fillable. \end{proposition}

Before proving Proposition \ref{prop-dondivnotfill}, we begin with a well-known observation regarding domination by stably parallelizable manifolds, which dates back to the work of Thom \cite{Thom} on the Steenrod realization problem for rational homology classes.

\begin{proposition}\label{prop-stabparadom}Every manifold is dominated by a stably parallelizable manifold.\end{proposition}

\begin{proof}Let $X^m$ be a manifold. Consider the {stable Hurewicz homomorphism} 
$$h : \pi_n^s(X) \to H_n(X),$$
where $\pi_n^s(X) = \varinjlim_k \pi_{n+k}(\Sigma^k X)$ is the $n$-th stable homotopy group of $X$. By \cite[Proposition 5.23]{Hatcher}, $h$ is a rational isomorphism for all $n$. Therefore, there is some integer $d > 0$ such that $d \cdot [X] \in H_m(X)$ is in the image of $h$. Let $\phi : S^{m+k} \to \Sigma^k X$ be such that $h([\phi]) = d[X]$. We would like to make $\phi$ ``transverse" to the equator $X \subset \Sigma^k X$. Then, $\phi : \phi^{-1}(X) \to X$ would be a degree $d$ map. Therefore, $\phi$ would be the required domination. 

Since $\Sigma^k X$ is not a manifold, we must carry out this step carefully. Consider the following homotopy-model of $\Sigma^k X$. Fix points $p \in S^k$, $q \in X$. Let us define the space $Z$ obtained from $S^k \times X \times (-2, 2)$ by collapsing $S^k \times \{q\} \times \{1\}$ to a point $*_1$ and $\{p\} \times X \times \{-1\}$ to a point $*_2$. Then $Z$ is homotopy equivalent to $\Sigma^k X$, and $Z$ is a stratified space with two singular point-strata $\{*_1, *_2\}$. Let $\pi : S^k \times X \times (-2, 2) \to Z$ denote the quotient map. Notice $\pi(\{p\} \times X \times \{1/2\})$ is an embedded copy of $X$ in $Z$ disjoint from the singular point-strata. By a slight abuse of notation, we shall denote this subspace as simply $X \subset Z$.

Composing $\phi$ with this homotopy equivalence $\Sigma^k X \simeq Z$, one arrives at a map $\phi : S^{m+k} \to Z$. Let $U := S^{n+k} \setminus \phi^{-1}(\{*_1, *_2\})$. Then $\phi|_U : U \to Z \setminus \{*_1, *_2\}$ is continuous proper map between smooth (noncompact) manifolds, and thus can be approximated by a smooth proper map. Hence, we assume without loss of generality that $\phi|_U : U \to Z \setminus \{*_1, *_2\}$ is indeed smooth. Since $X \subset Z \setminus \{*_1, *_2\}$ is compact, by a small perturbation if necessary, we may ensure $\phi$ is transverse to $X \subset Z \setminus \{*_1, *_2\}$. Let $Y := \phi^{-1}(X)$. Then $\dim Y = m$, and $\phi|_Y : Y \to X$ defines a degree $d$ map. Moreover, since $X \subset Z \setminus \{*_1, *_2\}$ has trivial normal bundle, so does $Y \subset S^{m+k}$. Therefore, $Y$ is stably parallelizable, as desired.
\end{proof}

We are now ready to prove Proposition \ref{prop-dondivnotfill}.

\begin{proof}[Proof of Proposition \ref{prop-dondivnotfill}]
Let $X^{2n+1}$ be a rationally essential manifold. By Proposition \ref{prop-stabparadom}, there exists a stably parallelizable manifold $M^{2n+1}$ dominating $X$. Therefore, $M$ is also rationally essential. Since $M$ is stably parallelizable, there exists $k \gg 1$ such that $M \times T^{2k}$ is parallelizable. Choose a co-rank $1$ trivial subbundle $\xi_0 \subset T(M \times T^{2k})$ and equip $\xi_0 \cong \underline{\mathbf C}^{n+k}$ with the obvious fiberwise complex structure. By the existence $h$-principle of Borman, Eliashberg and Murphy \cite{BEM}, $M \times T^{2k}$ admits a (overtwisted) contact structure $\xi$ in the homotopy class of $\xi_0$.

Let $\pi_0 : M \times T^{2k} \to M$ be the projection to the $M$ factor, and $\pi_j : M \times T^{2k} \to T^2$ denote the projection to the $j$-th torus factor, where $1 \leq j \leq k$. We shall prove by induction that there is a descending chain of asymptotically contact-holomorphic divisors
$$(M \times T^{2k}, \xi) \supset (Z_1, \xi|_{Z_1}) \supset (Z_2, \xi|_{Z_2}) \supset \cdots \supset (Z_k, \xi|_{Z_k}),$$
such that $f_j : (\pi_0, \pi_1, \cdots, \pi_{k-j}) : Z_j \to M \times T^{2(k-j)}$ is map of degree $1$. 

To this end, let $L$ be a line bundle over $T^2$ with Euler class $1$. In the base case of the induction, we consider the line bundle $L_0 := (\pi_k)^* L$. Using Theorem \ref{thm-imp}, we find an asymptotically contact-holomorphic divisor $Z_1 \subset (M \times T^{2k}, \xi)$ with $[Z_1] = \mathrm{PD}(c_1(L_0))$. Since $\langle c_1(L_0), [T_k] \rangle = 1$, the oriented intersection number between $Z_1$ and $T_k$ must be $1$. Therefore, 
$$f_1 := (\pi_0, \pi_1, \cdots, \pi_{k-1}) : Z_1 \to M \times T^{2(k-1)}$$
is a degree $1$ map. For the inductive step, suppose we have already constructed the desired contact submanifold $Z_j$, so that $f_j : Z_j \to M \times T^{2(k-j)}$ is a degree $1$ map. For clarity of exposition, let us denote $N := M \times T^{2(k-j-1)}$, so that 
$$M \times T^{2(k-j)} = N \times T^2$$ 
Consider the line bundle on $N \times T^2$ obtained by pulling back $L$ along the projection map $N \times T^2 \to T^2$. By a temporary abuse of notation, let us also denote the resulting line bundle on $N \times T^2$ as $L$. Let $L_j := f_j^* L$. Using Theorem \ref{thm-imp}, we find an asymptotically contact-holomorphic divisor $Z_{j+1} \subset (Z_j, \xi|_{Z_j})$ with $[Z_j] = \mathrm{PD}(c_1(L_j))$. Choose $x \in N$ such that $f_j$ is transverse to $\{x\} \times T^2 \subset N \times T^2$. Let us denote $T := \{x\} \times T^2$ for ease of notation. Then, 
\begin{align*}\langle c_1(L_j), [f_j^{-1}(T)] \rangle
= \langle f_j^* c_1(L), [f_j^{-1}(T)] \rangle 
&= \langle c_1(L), (f_j)_* [f_j^{-1}(T)] \rangle  \\
&= \langle c_1(L), (f_j)_* (f_j)_! [T] \rangle \end{align*}
Here, $(f_j)_! : H_2(N \times T^2) \to H_2(Z_j)$ denotes the Umkehr map in homology. Since $f_j$ is degree $1$, $(f_j)_* \circ (f_j)_!$ is multiplication by $1$ (see, \cite[Proposition 2.6]{dlH}). Thus, $(f_j)_* (f_j)_! [T] = [T]$. Since $\langle c_1(L), [T] \rangle = 1$, we conclude that $\langle c_1(L_j), [f_j^{-1}(T)] \rangle = 1$. Therefore, the oriented intersection number between $Z_j$ and $f_j^{-1}(\{x\} \times T^2)$ must be $1$. Consequently, 
$$f_{j+1} : Z_{j-1} \hookrightarrow Z_j \stackrel{f_j}{\to} M \times T^{2(k-j)} \to M \times T^{2(k-j-1)}$$
is a degree $1$ map. This proves the inductive step. 

In particular, this construction produces a contact manifold $(Z_{k-1}, \xi|_{Z_{k-1}})$ and an asymptotically contact-holomorphic divisor $Z_k \subset (Z_{k-1}, \xi|_{Z_{k-1}})$ such that there is a degree $1$ map $f_k : Z_k \to M$. Since $M$ is rationally essential, $Z_k$ is also rationally essential. Therefore, $Z_k$ is not Weinstein-fillable by Corollary \ref{cor-steinnotdom}. This concludes the proof.\end{proof}

We close this section with the following question. 

\begin{question}
    Let $(Y^{2n+1}, \xi)$ be a contact manifold and $L \to Y$ be a non-trivial complex line bundle. Are asymptotically contact-holomorphic divisors corresponding to $L$ always tight? Are they tight if moreover $(Y^{2n+1}, \xi)$ is tight?
\end{question}

\bibliography{cdbibs}
\bibliographystyle{alpha}
\end{document}